\documentclass[]{article}

\usepackage[margin=2cm]{geometry}
\usepackage{latexsym}		
\usepackage{graphicx}
\usepackage{color}
\usepackage{subfigure}
\usepackage{amsmath,amsfonts,wasysym,amsthm,amssymb,}
\usepackage{bm}
\usepackage{longtable}
\usepackage{tabularx}
\usepackage{comment}
\usepackage{float}
\usepackage{tikz}
\usepackage{tikz-3dplot}
\usetikzlibrary{positioning}
\usetikzlibrary{shapes}
\usetikzlibrary{plotmarks}
\usetikzlibrary{shapes.geometric, calc}

\newtheorem{theorem}{Theorem}
\newtheorem{lemma}{Lemma}

\newcommand{\vq}{ {\vec q } }
\newcommand{\vv}{ {\vec v } }
\newcommand{\vn}{ {\vec n } }

\numberwithin{equation}{section}


\title{Weak Galerkin finite element method for Poisson's equation on polytopal meshes with arbitrary small edges or faces}

\author{Qingguang Guan \\
	\\
	Center for Computation and Technology\\
	Louisiana  State University,\\
	Baton Rouge, LA 70803, USA}
\date{}
\begin{document}
	\maketitle	
\begin{abstract}
In this paper, the weak Galerkin finite element method for second order elliptic problems employing polygonal or polyhedral meshes with arbitrary small edges or faces was analyzed. With the shape regular assumptions, optimal convergence order for $H^1$ and $L_2$ error estimates were obtained. Also element based and edge based error estimates were proved.
\end{abstract}


\section{Introduction}\label{introd}
The weak galerkin finite element method using triangulated meshes was proposed by J. Wang, etc, see \cite{JW13}. Since then, the method gained applications in multiple areas, see \cite{Lin13, Lin14, Guan18}. A weak Galerkin mixed finite element method for second order elliptic problems has been successfully applied to polytopal meshes in \cite{Wang14}. The method was further 
developed in \cite{Lin15}, which was no longer a mixed method and provided an elegant and reliable way to solve second order elliptic problems. However, in \cite{Wang14} and \cite{Lin15}, the shape regularity requires the length of each edge or the area of each face being proportional to the diameter of the polygon or polyhedron which is the element of the partition. 
To get more flexibility of generating polytopal meshes, in this paper, we presented the new shape regularity assumptions and additional analysis to extend the weak Galerkin finite element method to use polytopal meshes with arbitrary small edges or faces.

We define the $L_2$ norm as $\|\cdot\|_{L_2(\Omega)},$ the inner product as $(\cdot,\cdot)_{\Omega}$, and the vector-valued space $H({\rm div};\Omega)$ as
$$
H({\rm div};\Omega) = \left\{ \vv: \vv\in [L_2(\Omega)]^n, \nabla\cdot\vv\in L_2(\Omega)\right\}.
$$ 
The crucial part of weak Galerkin finite element method for second order problem is the definition of a weak gradient operator and its approximation. Suppose we have a polygonal or polyhedral domain $D \subset \mathbb{R}^n,  (n=2,3),$ with interior $D_0$ and boundary $\partial D$ and a ``weak function'' on $D$ as $v=(v_0,v_b)$ with $v_0\in L_2(D_0)$  on $D_0$, $v_b\in L_2(\partial D),$  on $\partial D.$ And $v_b$ is not necessarily associated with the trace of $v_0$ on $\partial D.$  Then we denote the space $W(D)$ as
\begin{equation}\label{wd}
W(D) := \left\{v = (v_0,v_b): v_0\in L_2(D_0), v_b\in L_2(\partial D)\right\}.
\end{equation}
For any $v\in W(D)$, the weak gradient of $v$ is defined as a linear functional $\nabla_{w}v$
in the dual space of $H({\rm div};D)$ whose action on $\vq \in H({\rm div};D)$ is
\begin{equation}\label{wg1}
(\nabla_{w}v,\vq\ )_D := -\int_D v_0\nabla\cdot\vq\ {\rm d}x
+
\int_{\partial D} v_b \vq\cdot\vn\ {\rm d}S,
\end{equation}
where $\vn$ is the outward normal direction to $\partial D.$ By trace theorem, we know that the definition of $\nabla_{w} v$ is well posed and $\nabla_{w} v = \nabla v$ if $v\in H^1(D).$

The discrete weak gradient operator is defined with a polynomial subspace of  $H({\rm div};D)$.
For any integer $k\geq 0,$ $\mathbb{P}_k(D)$ is the polynomial space with degree no more than $k$. Denote the vector-valued space $[\mathbb{P}_k(D)]^n,$ then the discrete weak gradient $\nabla_{w,k,D} v$ of $v\in W(D)$ is defined as the solution of the following equation
\begin{equation}\label{wg2}
(\nabla_{w,k,D}v,\vq_k\ )_D = -\int_D v_0\nabla\cdot\vq_k\ {\rm d}x
+
\int_{\partial D} v_b \vq_k\cdot\vn\ {\rm d}S,\quad \forall \vq_k\in [\mathbb{P}_k(D)]^n,
\end{equation} 
where  $\nabla_{w,k,D} v \in [\mathbb{P}_k(D)]^n$, the definition is also well posed.

The paper is organized as follows: With the techniques from \cite{Lin15, Brenner17,  Guan184,  Brenner17-2}, several useful Lemmas were proved in Section 2 to 4. In Section 2, the new shape regularity assumptions were given and the $L_2$ operators were defined. In Section 3, we denoted the weak norm and the discrete weak Galerkin finite element space. Also, the error estimates for the $L_2$ operators were obtained.   In Section 4, the weak Galerkin finite element method was applied to Poisson's equation. The $H^1$ and $L_2$ error estimates were proved being optimal. In Section 5, we draw some conclusions.
\section{Shape Regularity}
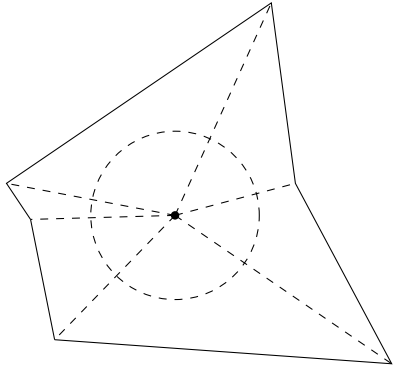
\begin{figure}[H]
	\begin{center}
		\begin{tikzpicture}[scale = 1.6]
		\coordinate (A) at (0.2,0.2);
		\coordinate (B) at (3,0);
		\coordinate (C) at (2.2,1.5);
		\coordinate (D) at (2,3);
		\coordinate (E) at (-0.2,1.5);
		\coordinate (F) at (0,1.2);
		\coordinate (O) at ($1/6*(A)+1/6*(B)+1/6*(C)+1/6*(D)+1/6*(E)+1/6*(F)$);
		\draw (A)
		--(B)
		--(C)
		--(D)
		--(E)
		--(F);
		\draw (F)--(A);
		\draw[style=dashed](O) circle (0.7);
		\fill [black] (O) circle (1pt);
		\draw[style=dashed](O)--(A);
		\draw[style=dashed](O)--(B);
		\draw[style=dashed](O)--(C);
		\draw[style=dashed](O)--(D);
		\draw[style=dashed](O)--(E);
		\draw[style=dashed](O)--(F);
		\end{tikzpicture}
	\end{center}\caption{A star shaped sub-domain $D$}\label{fg1}
\end{figure}
Let $\mathcal{T}_h$ be a partition of the domain $\Omega$ consisting of polygons in two dimensional space or polyhedrons
in three dimensional space. Denote by $\mathcal{E}_h$ the set of all edges or flat faces in $\mathcal{T}_h$ , and let $\mathcal{E}_h^0 = \mathcal{E}_h/ \partial \Omega$ be
the set of all interior edges or flat faces. For every element $D \in \mathcal{T}_h$ , we denote by
$|D|$ the area or volume of $D$ and by $h_D$ its diameter. Similarly, we denote by $|e|$ the
length or area of e and by $h_e$ the diameter of edge or flat face $e\in \mathcal{E}_h$ . We also set
as usual the mesh size of $\mathcal{T}_h$ by
$$h = \max\limits_{D \in \mathcal{T}_h} h_D.$$
All the elements of $\mathcal{T}_h$ are assumed to be closed and simply connected polygons
or polyhedrons; see Figure \ref{fg1} for an example in two dimensional space. We need some shape regularity assumptions for the partition $\mathcal{T}_h$
described as below.

Here the shape regularity assumptions are the same as in \cite{Brenner17-2}. Let $ {D}$ be the polygon or polyhedron with diameter $h_D$. Assume that 
\begin{equation}\label{assume1}
{D}\ {\rm is\ star\ shaped\ with\ respect\ to\ a\ disc/ball\ } \mathfrak{B}_D\subset D {\rm \ with\ radius\ =\ } \rho_D h_D,\ 0<\rho_D<1.
\end{equation}
Then we denote $\tilde{\mathfrak{B}}_D$ the  disc/ball concentric with $\mathfrak{B}_D$ whose radius is $h_D$. It's clear that 
\begin{equation}\label{assume2}
\mathfrak{B}_D\subset D\subset \tilde{\mathfrak{B}}_D.
\end{equation}
We will use the notation $A \apprle B$ to represent the inequality $A \leq (constant)B$.
The notation $A \approx B$ is equivalent to $A \apprle B$ and $A \apprge B$.
Figure \ref{fg1} is an example of $D$ satisfies the shape regularity assumptions. Based on the shape regularity assumptions \eqref{assume1} and \eqref{assume2}, there are several Lemmas. The hidden constants only depends on $\rho_D$ if there is no other statements.
\begin{lemma}\label{bramble}\cite{Bramble70} Bramble-Hilbert Estimates. Conditions \eqref{assume1}-\eqref{assume2} imply that we have the following estimates:
	\begin{equation}
	\inf\limits_{q\in\mathbb{P}_l} |\xi - q|_{H^m(D)} \apprle h^{l+1-m}
	|\xi|_{ H^{l+1}(D)}, \ 
	\forall \xi\in H^{l+1}(D),\ l = 0,\cdots, k,\ and\ 0\leq m \leq l.
	\end{equation}
\end{lemma}
Details can be found in \cite{Brenner07}, Lemma 4.3.8.
\subsection{A Lipschitz Isomorphism between $D$ and $\mathfrak{B}_D$}
With the star-shaped assumption \eqref{assume1}, there exists a  Lipschitz isomorphism
$\Phi: \mathfrak{B}_D\rightarrow D$ such that both $|\Phi|_{W^{1,\infty}(\mathfrak{B}_D)}$ and $|\Phi|_{W^{1,\infty}(D)}$ are bounded by constant that only depends on $\rho_D$ (see \cite{V11}, Section 1.1.8).
It then follows that 
\begin{equation}\label{Dh}
|D|\approx h_D^n {\ \rm and \ } |\partial D|\approx h_D^{n-1},\ n=2,3,
\end{equation}
where $|D|$ is the area of $D$ ($n=2$) or the volume of $D$ ($n=3$), and $|\partial D|$ is the arclength of $\partial D$ or the surface area of $D$ ($n=3$). Moreover from Theorem 4.1 in \cite{J87}, we have 
\begin{eqnarray}
\|\xi\|_{L_2(\partial D)} &\approx& \|\xi\circ\Phi\|_{L_2(\partial \mathfrak{B}_D)},\quad \forall \xi\in L_2(\partial D), \label{iso1} \\
\|\xi\|_{L_2(  D)} &\approx& \|\xi\circ\Phi\|_{L_2( \mathfrak{B}_D)},\quad \forall \xi\in L_2(  D), \label{iso2}  \\
|\xi|_{H^1(D)} &\approx& |\xi\circ\Phi|_{H^1(  \mathfrak{B}_D)},\quad \forall \xi\in H^1(D). \label{iso3} 
\end{eqnarray}
Same as in \cite{Brenner17-2}, from \eqref{Dh}, \eqref{iso1}-\eqref{iso3} and the standard (scaled) trace inequalities for $H^1(\mathfrak{B}_D)$ we have
\begin{lemma}\label{trace}
(Trace Inequality (2.18) ) \cite{Brenner17-2}. 
Let $\mathcal{T}_h$ be a partition of the domain $\Omega$ into polygons ($n=2$) or polyhedrons ($n=3$). Assume that $D\in \mathcal{T}_h$ satisfies the assumptions \eqref{assume1} and \eqref{assume2} as specified above. Then we have
$$
h_D^{-1}\|\xi\|_{L_2(\partial D)}^2 \apprle  h_D^{-2}\|\xi\|_{L_2(D)}^2 + |\xi |_{H^1(D)}^2 ,
$$
for any $\xi \in H^1(D ).$
\end{lemma}
\subsection{$L_2$ Projection Operators}
For each element $D \in \mathcal{T}_h$, denote by $Q_{k,D}^0$ the $L_2$ projection from $L_2 (D )$ onto $\mathbb{P}_k (D ).$
Analogously, for each edge or flat face $e \in \mathcal{E}_h$, let $Q_{k,D}^b$ be the $L_2$ projection operator
from $L_2 (e)$ onto $\mathbb{P}_k (e)$. 
We define a projection operator $Q_h$ as follows
\begin{equation}\label{l2proj}
Q_h v|_D := (Q_{k,D}^0 v_0 , Q_{k,D}^b v_b),\
\forall v \in W(D).
\end{equation}

Denote by $\mathbb{Q}_{k-1,D}$ the $L_2$ projection from $[L_2(D)]^n$ onto the local discrete
gradient space $[\mathbb{P}_{k-1} (D)]^n$. 

With these definitions, we also have the following Lemmas.
\begin{lemma}\label{l_discrete}
	For any $\vq\in [\mathbb{P}_k(D)]^n$, we have
	$$
	h_D\|\vq\|^2_{L_{2}(\partial D)}+h_D^{2}\|\nabla\cdot \vq\|^2_{L_2(D)}
	\apprle 
	\|\vq\|^2_{L_2(D)},
	$$	
	the hidden constant only depends on $\rho_D$ and $k$.	
\end{lemma}
\begin{proof}
Suppose $n=2$, and $\vq = (q_1,q_2)$, then by Lemma 2.3 in \cite{Brenner17-2}, we have
$$
	h_D\|q_i\|^2_{L_{2}(\partial D)}+h_D^{2} | q_i|^2_{H^1(D)}
\apprle 
\|q_i\|^2_{L_2(D)},\ i = 1,2,
$$
so that 
	$$
h_D\|\vq\|^2_{L_{2}(\partial D)}+h_D^{2}\|\nabla\cdot \vq\|^2_{L_2(D)}
\apprle 
\|\vq\|^2_{L_2(D)}.
$$	
For $n=3,$ the proof is similar.
\end{proof}
\begin{lemma}\label{l4}
	Lemma 3.9 \cite{Brenner17-2}.  Assume that $D$ satisfies all the assumptions
	as specified above. Then, we have
	$$
	 |Q_{k,D}^0 \xi|_{H^1(D)} \apprle | \xi |_{H^1(D)},\ \forall \xi\in H^1(D),
	$$
	the hidden constant only depends on $\rho_D$ and $k$.
\end{lemma}
\begin{lemma}\label{leq}
Lemma 5.1 in \cite{Lin15}. Let $Q_h$ be the projection operator defined as in \eqref{l2proj}. Then, on each
element $D \in \mathcal{T}_h$, we have
$$\nabla_{w,k-1,D}(Q_h \xi) = \mathbb{Q}_{k-1,D} (\nabla\xi),\ 
\forall \xi\in H^1(D).$$
\end{lemma}
The following lemma provides some estimates for the projection operators $Q_h$ and
$\mathbb{Q}_{k-1,D}.$
\begin{lemma}\label{ler}
Let $D$ satisfy the shape regular
assumptions as given above. Then for $\xi\in H^{k+1}(D)$, we have
\begin{equation}\label{l2Q1}
\|\xi-Q_{k,D}^0 \xi\|_{L_2(D)}^2+
h_D^2 |\xi-Q_{k,D}^0 \xi |_{H^1(D)}^2
\apprle  h^{2(k+1)}\|\xi\|_{H^{k+1}(D)}^2,
\end{equation}
\begin{equation}\label{l2Q2}
\|
\nabla\xi - \mathbb{Q}_{k-1,D}\nabla\xi
\|_{L_2(D)}^2
+
h_D^2|\nabla\xi - \mathbb{Q}_{k-1,D}\nabla\xi|_{H^1(D)}^2
\apprle h^{2k}\|\xi\|_{H^{k+1}(D)}^2,
\end{equation}
the hidden constant only depends on $\rho_D$ and $k$.
\end{lemma}
\begin{proof}
For the $L_2$ projection $Q_{k,D}^0$ in (\ref{l2Q1}), with Lemma \ref{bramble}, we have
$$
\|\xi-Q_{k,D}^0 \xi\|_{L_2(D)}^2
\apprle  h^{2(k+1)}\|\xi\|_{H^{k+1}(D)}^2,
$$
Let $p$ be any polynomial  with degree $k$, with Lemma \ref{bramble} and Lemma \ref{l4}, we have
\begin{eqnarray*}
 | \xi-Q_{k,D}^0 \xi |_{H^1(D)}
&\leq&
 | \xi- p |_{H^1(D)}+ | p-Q_{k,D}^0\xi |_{H^1(D)}\\
&\leq&
 | \xi- p |_{H^1(D)}+  | Q_{k,D}^0 (p-\xi) |_{H^1(D)}\\
&\apprle&
 | \xi- p  |_{H^1(D)}\\
&\apprle&
h^{k}\|\xi\|_{H^{k+1}(D)}.
\end{eqnarray*}
For the $L_2$ projection $\mathbb{Q}_{k-1,D}$ in \eqref{l2Q2}, with Lemma \ref{bramble}, suppose $\nabla\xi = (\xi_x,\xi_y)$ for $n=2$, we have
\begin{eqnarray*}
\|
\nabla\xi - \mathbb{Q}_{k-1,D}\nabla\xi
\|_{L_2(D)}
&\apprle& 
\|
\xi_x - Q_{k-1,D}^0\xi_x
\|_{L_2(D)}+\|
\xi_y - Q_{k-1,D}^0\xi_y
\|_{L_2(D)}
\\
&\apprle& h^{k}\|\xi\|_{H^{k+1}(D)}.
\end{eqnarray*}	
Then we consider the second term in \eqref{l2Q2} with Lemma \ref{bramble} and Lemma \ref{l4},
\begin{eqnarray*}
|\nabla\xi - \mathbb{Q}_{k-1,D}\nabla\xi|_{H^1(D)}^2
&=&
|\xi_x - Q_{k-1,D}^0\xi_x|_{H^1(D)}^2
+
|\xi_y - Q_{k-1,D}^0\xi_y|_{H^1(D)}^2\\
&\apprle&
h^{2(k-1)}\|\xi\|_{H^{k+1}(D)}^2.
\end{eqnarray*}	
The case $n=3$ is similar. So that \eqref{l2Q1} and \eqref{l2Q2} are proved.
\end{proof}
\section{The Weak Galerkin Finite Element Scheme}
Suppose 
$\Omega$ is a bounded convex polygonal or polyhedral domain in $\mathbb{R}^n, (n=2,3)$,
$\mathcal{T}_h$ is a shape regular partition of $\Omega$. On each $D\in \mathcal{T}_h,$ we have $W(D)$ defined in \eqref{wd}. Then let $W$ be the weak functional space on $\mathcal{T}_h$ as
$$
W := \prod_{D\in \mathcal{T}_h} W(D).
$$
Same as Section 4.2 in \cite{Lin15}, we denote $V$ as a subspace of $W$. For each interior edge or face $e \in \mathcal{E}_h^0$, there are $D_1$ and $D_2$, so that $e\subset \partial D_1\cap \partial D_2$. Denote $v\in V$, so that for $v_i\in W(D_i), i=1,2,$  we have 
$$
v_1|_e = v_2|_e.
$$
Then the weak norm of $v\in V$ is defined as
\begin{equation}\label{w1n}
| v |^2_{k-1,w} = \sum\limits_{D\in\mathcal{T}_h}
\int_{D}\nabla_{w,k-1,D} v\cdot\nabla_{w,k-1,D}  v\ {\rm d} x
+ 
h_D^{-1}
\langle  
v_0-v_b, v_0-v_b
\rangle_{\partial D},
\end{equation}
where $k\geq 1$ is integer.
Let $\mathbb{P}_k(D_0)$ be the set of polynomials on $D_0$ with degree no more than $k$, and $\mathbb{P}_k(e)$ be the set of polynomials on each edge or face $e\in \mathcal{E}_h$. Then the weak finite element space is given by
\begin{equation}\label{Sjl}
V_h:= \{v: v|_{D_0}\in \mathbb{P}_k(D_0)\ \forall D\in \mathcal{T}_h \ {\rm and}\ v|_{e}\in \mathbb{P}_k(e) \ \forall e\in \mathcal{E}_h \}.
\end{equation}
Denote the space $V^0_h$ as a subspace of $V_h$ which has vanishing boundary value on $\partial\Omega$ by
\begin{equation}\label{Sjl0}
V^0_h := \{v: v\in V_h \ {\rm and}\ v|_{\partial\Omega} =0  \}.
\end{equation}
\begin{lemma}\label{lem_h1}
Assume that $\mathcal{T}_h$ is shape regular. Then for any $w \in H^{k+1} (\Omega)$ and
$v = (v_0, v_b) \in V_h$, we have
\begin{equation}\label{e-trace1}
\left|
\sum\limits_{D\in \mathcal{T}_h} h_D^{-1}
\langle  
Q_{k,D}^0 w-Q_{k,D}^b w , 
v_0-v_b
\rangle_{\partial D}
\right|\apprle h^k\|w\|_{H^{k+1}(\Omega)}|v|_{k-1,w},
\end{equation}
\begin{equation}\label{e-trace2}
\left|
\sum\limits_{D\in \mathcal{T}_h}
\langle  
(\nabla w - \mathbb{Q}_{k-1,D} \nabla w)\cdot\vn,   
v_0-v_b
\rangle_{\partial D}
\right|\apprle   h^k\|w\|_{H^{k+1}(\Omega)}|v|_{k-1,w},
\end{equation}
where $k\geq 1$ and the hidden constant only depends on $\rho_D$ and $k$. 
\end{lemma}
\begin{proof}
The proof is similar as Lemma 5.3 in \cite{Lin15}. For completeness, we give the proof here.

To get \eqref{e-trace1}, we have
\begin{eqnarray*}
\left|
\sum\limits_{D\in \mathcal{T}_h} h_D^{-1}
\langle  
Q_{k,D}^0 w-Q_{k,D}^b w , 
v_0-v_b
\rangle_{\partial D}
\right|
&=&
\left|
\sum\limits_{D\in \mathcal{T}_h} h_D^{-1}
\langle  
Q_{k,D}^0 w- w , 
v_0-v_b
\rangle_{\partial D}
\right|
\\
&\apprle&
\left(
\sum\limits_{D\in \mathcal{T}_h} h_D^{-1}
\| Q_{k,D}^0 w- w \|_{L_2(\partial D)}^2
\right)^{1/2}
\left(
\sum\limits_{D\in \mathcal{T}_h} h_D^{-1}
\| v_0-v_b \|_{L_2(\partial D)}^2
\right)^{1/2},
\end{eqnarray*}
with Lemma \ref{trace} and Lemma \ref{ler}, \eqref{e-trace1} is obtained.

To get \eqref{e-trace2}, we have
\begin{eqnarray*}
\left|
\sum\limits_{D\in \mathcal{T}_h}
\langle  
(\nabla w - \mathbb{Q}_{k-1,D} \nabla w)\cdot\vn,   
v_0-v_b
\rangle_{\partial D}
\right|
\apprle
	\left(
	\sum\limits_{D\in \mathcal{T}_h} h_D 
	\|\nabla w - \mathbb{Q}_{k-1,D} \nabla w\|_{L_2(\partial D)}^2
	\right)^{1/2}
	\left(
	\sum\limits_{D\in \mathcal{T}_h} h_D^{-1}
	\| v_0-v_b \|_{L_2(\partial D)}^2
	\right)^{1/2},
\end{eqnarray*}
with Lemma \ref{trace} and Lemma \ref{ler}, \eqref{e-trace2} is obtained.
\end{proof}
\section{The Weak Galerkin Finite Element Method for Poisson's Equation}
Let $\Omega$ be a bounded convex polygonal or polyhedral domain in $\mathbb{R}^n, (n=2,3)$, $f\in L_2(\Omega)$, the Poisson's equation is
\begin{equation}\label{poisson-equation}
\begin{cases}
\ -\Delta u            &= f, \\
\ u|_{\partial \Omega} &= 0.
\end{cases}
\end{equation}
And $\forall v\in V_h,$ the weak gradient of $v$ is defined on each element $D$ by \eqref{wg2},   respectively. 
And for any $u,v\in V_h$, the bilinear form is defined as
\begin{equation}\label{bhr} 
a_{h}(u,v) = \sum\limits_{D\in\mathcal{T}_h}\int_{D}\nabla_{w,k-1,D} u\cdot\nabla_{w,k-1,D} v\ {\rm d} x.
\end{equation}
The stabilization term is:
\begin{equation}\label{stable} 
s_{h}(u,v) = \sum\limits_{D\in\mathcal{T}_h} 
h_D^{-1}
\langle  
u_0-u_b, v_0-v_b
\rangle_{\partial D}.
\end{equation}
A numerical solution for \eqref{poisson-equation} can be obtained by seeking $u_h =(u_0,u_b)\in V^0_h$ such that 
\begin{equation}\label{num-poisson}
a_s(u_h,v_h) = a_h(u_h,v_h)+s_h(u_h,v_h) = (f,v_0)_{\Omega}, \ \forall v=(v_0,v_b)\in V_h^0.
\end{equation}
Same as Lemma 7.1 in \cite{Lin15}, we have the discrete Poincar$\acute{\rm e}$ inequality:
\begin{lemma}\label{poincare}
Suppose the partition $\mathcal{T}_h$ is shape regular. Then we have
$$
\|v_0\|_{L_2(\Omega)} \apprle |v|_{k-1,w}, \ \forall v=(v_0,v_b)\in V_h^0,
$$
the hidden constant only depends on $\rho_D$ and $k$.
\end{lemma}
Also, we have the existence and uniqueness of the solution of \eqref{num-poisson} with the same proof as Lemma 7.2 in \cite{Lin15}.

To get the error analysis, we need another Lemma.
\begin{lemma}\label{K-bound}
	Assume that $D$ satisfies all the assumptions as specified above. Then, we have
	\begin{equation}
	\|\nabla v_0\|^2_{L_2(D)}
	\apprle
	\|\nabla_{w,k-1,D} v\|^2_{L_2(D)}+
	h_D^{-1}\|v_b-v_0\|^2_{L_2(\partial D)},\  \forall v\in V_h|_D,
	\end{equation}
	the hidden constant only depends on $\rho_D$ and $k$.
\end{lemma}
\begin{proof}
	Suppose on $D$ we have $v=(v_0,v_b)$, $\vq \in H({\rm div};D)$ , then by the definition of $\nabla_{w,k-1,D}$, we have
	\begin{eqnarray*}
		(\nabla_{w,k-1,D} v,\vq )_D 
		&=& 
		-(v_0,\nabla\cdot\vq)_D +\langle v_b,\vq\cdot\vn \rangle_{\partial D}\\
		&=& 
		(\nabla v_0,\vq)_D +\langle v_b-v_0,\vq\cdot\vn \rangle_{\partial D}
	\end{eqnarray*}
	so that
	$$
	(\nabla_{w,k-1,D} v-\nabla v_0,\vq )_D = \langle v_b-v_0,\vq\cdot\vn \rangle_{\partial D}
	$$
	let $\vq = \nabla_{w,k-1,D} v-\nabla v_0$, then
	$$
	(\nabla_{w,k-1,D} v-\nabla v_0,\nabla_{w,k-1,D} v-\nabla v_0 )_D = \langle v_b-v_0,(\nabla_{w,k-1,D} v-\nabla v_0)\cdot\vn \rangle_{\partial D},
	$$
	By the discrete inequalities of polynomials, Lemma \ref{l_discrete}, we have
	$$
	\|\nabla_{w,k-1,D} v-\nabla v_0\|_{L_2(D)}\apprle
	h_D^{-1/2}\|v_b-v_0\|_{L_2(\partial D)},
	$$
	so that
	\begin{equation*}
	\|\nabla v_0\|^2_{L_2(D)}
	\apprle
	\|\nabla_{w,k-1,D} v\|^2_{L_2(D)}+
	h_D^{-1}\|v_b-v_0\|^2_{L_2(\partial D)}.
	\end{equation*}
\end{proof}
\subsection{Error Analysis}
Let $u$ be the solution of \eqref{poisson-equation} and $v\in V_h^0$. It follows from the definition of weak gradient \eqref{wg2}, and the integration by parts that
\begin{eqnarray}
(\nabla_{w,k-1,D} Q_hu,\nabla_{w,k-1,D}v)_D
=
 (\nabla u,\nabla v_0)_D
-
\langle v_0-v_b,(\mathbb{Q}_{k-1,D}\nabla u)\cdot\vn \rangle_{\partial D}. \label{wI_h}
\end{eqnarray}
Then, multiply \eqref{poisson-equation} by $v_0$ of $v=(v_0,v_b)\in V_h^0$ we have
\begin{equation}\label{wu}
\sum\limits_{D\in\mathcal{T}_h} (\nabla u,\nabla v_0)_D
=(f,v_0)+\sum\limits_{D\in\mathcal{T}_h}\langle  v_0-v_b , \nabla u\cdot\vn  \rangle_{\partial D}.
\end{equation}
Combine \eqref{wI_h} and \eqref{wu}, we have
\begin{equation}\label{w-error1}
\sum\limits_{D\in\mathcal{T}_h} 
(\nabla_{w,k-1,D} Q_hu,\nabla_{w,k-1,D}v)_D
=
(f,v_0)_{\Omega} 
+
\sum\limits_{D\in\mathcal{T}_h}
\langle  
v_0-v_b , 
(\nabla u - \mathbb{Q}_{k-1,D} \nabla u)\cdot\vn  
\rangle_{\partial D}.
\end{equation}
Adding $s_h({Q}_h u, v)$ to both sides of \eqref{w-error1} gives
\begin{equation}\label{w-error2}
a_s(Q_h u,v)
=
(f,v_0)_{\Omega}  
+
\sum\limits_{D\in\mathcal{T}_h}
\langle  
v_0-v_b , 
(\nabla u - \mathbb{Q}_{k-1,D} \nabla u)\cdot\vn  
\rangle_{\partial D} +s_h(Q_hu,v).
\end{equation}
Subtracting \eqref{num-poisson} from \eqref{w-error2}, we have the error equation
\begin{equation}\label{w-error3}
a_s(e_h,v)
=
\sum\limits_{D\in\mathcal{T}_h}
\langle  
v_0-v_b , 
(\nabla u - \mathbb{Q}_{k-1,D} \nabla u)\cdot\vn  
\rangle_{\partial D} +s_h(Q_hu,v).
\end{equation}
where
$$
e_h|_D = (e_0,e_b)|_D:= (Q_{k,D}^0 u -u_0, Q_{k,D}^b u -u_b)|_D= (u_h-Q_hu)|_{D}
$$
which is the error between the weak Galerkin finite element solution and the $L_2$ projection of the exact solution.
Then we define a norm $\|\cdot\|_h$ as
$$
\|\vv\|_h^2 := \sum\limits_{D\in\mathcal{T}_h} \|\vv\|_{L_2(D)}, \ \forall \vv\in [L_2(\Omega)]^n,
$$
and 
$$
(\nabla_{w,k-1} u_h)|_D = \nabla_{w,k-1,D} (u_h|_D);\quad (\mathbb{Q}_{k-1}(\nabla u))|_D = \mathbb{Q}_{k-1,D} (\nabla u|_D).
$$
\begin{theorem}\label{th1}
Let $u_h \in V_h$ be the weak Galerkin finite element solution of the
problem (4.1). Assume that the exact solution is so regular that
$u \in H^{k+1} (\Omega)$. Then we have
\begin{eqnarray}
\|\nabla u - \nabla_{w,k-1} u_h\|_h 
&\apprle& h^k\|u\|_{H^{k+1}(\Omega)}, \label{we1} \\
\|\nabla u - \nabla u_0\|_h 
&\apprle& 
 h^k\|u\|_{H^{k+1}(\Omega)}, \label{we2} 
\end{eqnarray}
the hidden constant only depends on $\rho_D$ and $k$.
\end{theorem}
\begin{proof}
Let $v = e_h$ in \eqref{w-error3}, with \eqref{lem_h1} we have
\begin{equation}
|e_h|_{k-1,w}^2
=
\sum\limits_{D\in\mathcal{T}_h}
\langle  
e_0-e_b , 
(\nabla u - \mathbb{Q}_{k-1,D} \nabla u)\cdot\vn  
\rangle_{\partial D} +s_h(Q_hu,e_h).
\end{equation}
It then follows from Lemma \ref{lem_h1} that
\begin{equation}\label{error1}
|e_h|_{k-1,w}^2 \apprle  h^k\|u\|_{H^{k+1}(\Omega)}|e_h|_{k-1,w}.
\end{equation}
Based on \eqref{error1}, firstly, we prove \eqref{we1}, 
\begin{eqnarray*}
\|\nabla u - \nabla_{w,k-1} u_h\|_h 
&\leq&
\|\nabla u - \mathbb{Q}_{k-1}(\nabla u)\|_h 
+\|\mathbb{Q}_{k-1}(\nabla u) - \nabla_{w,k-1} u_h\|_h, 
\end{eqnarray*}
with Lemma \ref{leq} and Lemma \ref{ler}, we have
$$
\|\mathbb{Q}_{k-1}(\nabla u) - \nabla_{w,k-1} u_h\|_h
=
\|\nabla_{w,k-1} ( {Q}_{h} u - u_h) \|_h\leq |e_h|_{k-1,w}
$$
so that we have \eqref{we1}.

Secondly, with Lemma \ref{K-bound}, we have
\begin{eqnarray*}
	\sum_{D\in\mathcal{T}_h}\|\nabla (Q^0_{k,D} u -u_h|_{D_0})\|^2_{L_2(D)} 
	&=&
	\sum\limits_{D\in\mathcal{T}_h} \|\nabla  e_0\|^2_{L_2(D)}\\
	&\apprle&
	\sum\limits_{D\in\mathcal{T}_h} \|\nabla_{w,k-1} e_h\|^2_{L_2(D)} +h_D^{-1}\|e_b-e_0\|^2_{L_2(\partial D)}\\
	&\apprle&
	|u_h-Q_h u|_{k-1,w}^2
\end{eqnarray*}
which means
$$
\sum_{D\in\mathcal{T}_h}\|\nabla (Q^0_{k,D} u -u_h|_{D_0})\|^2_{L_2(D)}  \apprle h^{2k}\|u\|^2_{H^{k+1}(\Omega)}.
$$
Also by
$$
\sum_{D\in\mathcal{T}_h}\|\nabla (Q^0_{k,D} u -u)\|^2_{L_2(D)}  \apprle h^{2k}\|u\|^2_{H^{k+1}(\Omega)},
$$
with Lemma \ref{ler}, we have \eqref{we2}
$$
\|\nabla u - \nabla u_0\|_h  \apprle h^k\|u\|_{H^{k+1}(\Omega)}.
$$
\end{proof}
With Lemma \ref{poincare}, we have \eqref{l2e} with the similar proof as Theorem 8.2 in \cite{Lin15}.  Also, we complete the proof of estimation for edge based error:
$$
\|u-u_h\|_{\mathcal{E}_h}^2 := \sum\limits_{e\in\mathcal{E}_h} h_e\int_e |u-u_b|^2 \rm{d} s ,
$$ 
where $u$ is the exact solution and $u_h$ is the numerical solution of \eqref{poisson-equation}.
\begin{theorem}\label{t2}
	Let $u_h \in V_h$ be the weak Galerkin finite element solution of the
	problem (4.1). Assume that the exact solution 
	$u \in H^{k+1} (\Omega)$. Then we have
	\begin{eqnarray}
	\|u - u_0\|_{L_2(\Omega)}
	&\apprle& h^{k+1}\|u\|_{H^{k+1}(\Omega)}, \label{l2e} \\
		\|u - u_h\|_{\mathcal{E}_h}
		&\apprle& h^{k+1}\|u\|_{H^{k+1}(\Omega)}, \label{l2eb}
	\end{eqnarray}
	the hidden constant only depends on $\rho_D$ and $k$.
\end{theorem}
\begin{proof}
Firstly, to prove \eqref{l2e}, we begin with a dual problem seeking $\phi\in H_0^1(\Omega)$ such that
$
-\Delta \phi = e_0.
$
With the assumption of $\Omega$, we have $\|\phi\|_{H^2(\Omega)}\apprle \|e_0\|_{L_2(\Omega)}$.

Then we have
\begin{equation}\label{l2e1}
\|e_0\|_{L_2(\Omega)}^2 = 
\sum\limits_{D\in\mathcal{T}_h}(\nabla\phi,\nabla e_0)_D 
-
\sum\limits_{D\in\mathcal{T}_h}
\langle 
\nabla\phi\cdot\vn, 
e_0-e_b 
\rangle_{\partial D}.
\end{equation}
Let $u=\phi$ and $v=e_h$ in \eqref{wI_h}, we have
\begin{equation}\label{l2e2}
(\nabla_{w,k-1,D} Q_h \phi,\nabla_{w,k-1,D} e_h)_D
=
(\nabla \phi,\nabla e_0)_D
-
\langle e_0-e_b,(\mathbb{Q}_{k-1,D}\nabla \phi)\cdot\vn \rangle_{\partial D}. 
\end{equation}
Combine \eqref{l2e1} and \eqref{l2e2}, we have
\begin{equation}\label{l2e3}
\|e_0\|_{L_2(\Omega)}^2 = 
(\nabla_{w,k-1} Q_h \phi,\nabla_{w,k-1} e_h)_\Omega
+
\sum\limits_{D\in\mathcal{T}_h}
\langle 
(\mathbb{Q}_{k-1,D}\nabla \phi-\nabla\phi)\cdot\vn, 
e_0-e_b 
\rangle_{\partial D}.
\end{equation}
Then let $v = Q_h \phi$ in \eqref{w-error3}, such that
\begin{equation}\label{l2e4}
(\nabla_{w,k-1} Q_h \phi,\nabla_{w,k-1} e_h)_\Omega 
=
\sum\limits_{D\in\mathcal{T}_h}
\langle  
Q^0_{k,D}\phi- Q^b_{k,D}\phi , 
(\nabla u - \mathbb{Q}_{k-1,D} \nabla u)\cdot\vn  
\rangle_{\partial D} +s_h(Q_hu,Q_h \phi) -s_h(e_h,Q_h \phi).
\end{equation}
Plugging \eqref{l2e4} in \eqref{l2e3}, we get the following equation
\begin{eqnarray}\label{l2e5}
\|e_0\|_{L_2(\Omega)}^2
&=& 
\sum\limits_{D\in\mathcal{T}_h}
\langle  
Q^0_{k,D}\phi- Q^b_{k,D}\phi , 
(\nabla u - \mathbb{Q}_{k-1,D} \nabla u)\cdot\vn  
\rangle_{\partial D}\nonumber \\ 
&&+
\sum\limits_{D\in\mathcal{T}_h}
\langle 
(\mathbb{Q}_{k-1,D}\nabla \phi-\nabla\phi)\cdot\vn, 
e_0-e_b 
\rangle_{\partial D}
+
s_h(Q_hu,Q_h \phi) -s_h(e_h,Q_h \phi).
\end{eqnarray}
By Lemma \ref{ler}, Lemma \ref{lem_h1}, Lemma \ref{poincare} and \eqref{error1}, we can estimate the right hand side terms of  \eqref{l2e5} same as in \cite{Lin15}.
Then we have
\begin{eqnarray*}
\|Q_k^0u - u_0\|_{L_2(\Omega)}
\apprle h^{k+1}\|u\|_{H^{k+1}(\Omega)},
\end{eqnarray*}
with 
$$
\|Q_k^0u - u\|_{L_2(\Omega)}
\apprle h^{k+1}\|u\|_{H^{k+1}(\Omega)},
$$
the error estimate \eqref{l2e} is obtained.

Then, to prove \eqref{l2eb}, we begin with the edge based error
\begin{eqnarray*}
\|u-u_h\|_{\mathcal{E}_h}^2 &:=& \sum\limits_{e\in\mathcal{E}_h} h_e\int_e |u-u_b|^2 \rm{d} s\\
&\apprle&
\sum\limits_{D\in\mathcal{T}_h} h_D\int_{\partial D} |u-u_b|^2 \rm{d} s\\
&\apprle&
\sum\limits_{D\in\mathcal{T}_h} h_D 
\left( 
  \langle u-u_0,   u-u_0   \rangle_{\partial D}
+2\langle u-u_0,   u_0-u_b \rangle_{\partial D}
+ \langle u_0-u_b, u_0-u_b \rangle_{\partial D}
\right)\\
&\apprle&
\sum\limits_{D\in\mathcal{T}_h} h_D 
\left( 
\|u-u_0 \|^2_{L_2(\partial D)}
+ \|u_0-u_b\|^2_{L_2(\partial D)}
\right).
\end{eqnarray*}
By Lemma \ref{trace}, \eqref{we2}, \eqref{l2e}, we have
\begin{eqnarray*}
h_D  
\|u-u_0 \|^2_{L_2(\partial D)}	
&\apprle&
\|u-u_0\|_{L_2(D)}^2 + h_D^2|u-u_0 |_{H^1(D)}^2,
\end{eqnarray*}	
so that 
$$
\sum\limits_{D\in\mathcal{T}_h} h_D 
\|u-u_0 \|^2_{L_2(\partial D)} \apprle h^{2(k+1)}\|u\|^2_{H^{k+1}(\Omega)}.
$$
Then, by Lemma \ref{trace}, \eqref{l2Q1}, \eqref{w1n},  \eqref{error1}, we have
\begin{eqnarray*}
h_D 
\|u_0-u_b\|^2_{L_2(\partial D)}	
&\apprle&
h_D\|u_0-Q_{k,D}^0u -(u_b-Q_{k,D}^b u)\|^2_{L_2(\partial D)}	
+
h_D\|Q_{k,D}^0u-Q_{k,D}^b u\|^2_{L_2(\partial D)}	\\
&\apprle&
h_D\|e_0 -e_b\|^2_{L_2(\partial D)}	
+
h_D\|Q_{k,D}^0u - u\|^2_{L_2(\partial D)},
\end{eqnarray*}
so that 
$$
\sum\limits_{D\in\mathcal{T}_h} h_D 
\|u_0-u_b\|^2_{L_2(\partial D)}	 \apprle h^{2(k+1)}\|u\|^2_{H^{k+1}(\Omega)}.
$$	
And \eqref{l2eb} is proved.
\end{proof}
\section{Conclusions}
The shape regularity assumptions here are different with the assumptions in \cite{Wang14} and \cite{Lin15}. To get the similar results, it requires new Lemmas from Lemma \ref{bramble} to Lemma \ref{lem_h1},  which are valid under the new shape regularity assumptions. Also we provide element based error estimation in \eqref{we2} of Theorem \ref{th1} and edge based error estimation in \eqref{l2eb} of Theorem \ref{t2} which make it easier to compare the numerical solutions with the exact ones.


\begin{thebibliography}{99} 
\bibitem{JW13}
J. Wang and X. Ye, A weak galerkin finite element method for second-order elliptic problems,
\emph{Journal of Computational and Applied Mathematics.} \textbf{241} (2013) 103--115.
\bibitem{Lin13}
Lin Mu, Junping Wang, Guowei Wei, Xiu Ye, and Shan Zhao. Weak Galerkin methods for second order elliptic interface problems. \emph{Journal of computational physics.} \textbf{250}, (2013): 106--125. 
\bibitem{Wang14}
Junping Wang, and Xiu Ye. A weak Galerkin mixed finite element method for second order elliptic problems. \emph{Mathematics of Computation.} \textbf{83}, no. 289 (2014): 2101--2126.
\bibitem{Lin14}
Mu Lin, Junping Wang, and Xiu Ye. Weak Galerkin finite element methods for the biharmonic equation on polytopal meshes. \emph{Numerical Methods for Partial Differential Equations.} \textbf{30}, (2014): 1003--1029. 
\bibitem{Lin15}
Mu Lin, Junping Wang, and Xiu Ye. Weak Galerkin finite element methods on polytopal meshes.  \emph{International Journal of Numerical Analysis $\&$ Modeling.}  \textbf{12}, Issue 1 (2015): 31--53.
\bibitem{Guan18}
Qingguang Guan, Max Gunzburger, and Wenju Zhao. Weak-Galerkin finite element methods for a second-order elliptic variational inequality. \emph{Computer Methods in Applied Mechanics and Engineering.} \textbf{337}, (2018): 677--688.
\bibitem{Bramble70}
J.H. Bramble and S.R. Hilbert. Estimation of linear functionals on Sobolev spaces with applications to
Fourier transforms and spline interpolation.  \emph{SIAM J. Numer. Anal.}, \textbf{7}, (1970): 113--124, .
\bibitem{Brenner07}
S.C. Brenner and L.R. Scott. The Mathematical Theory of Finite Element Methods (Third Edition).
\emph{Springer-Verlag}, New York, 2008.
\bibitem{Brenner17}
S.C. Brenner, Q. Guan, and L.-Y. Sung. On some estimates for virtual element methods, {\em Computational Methods in Applied Mathematics}, \textbf{4}, (2017): 553--574.
\bibitem{Guan184}
Qingguang Guan, Some Estimates of Virtual Element Methods for Fourth Order Problems, \emph{arXiv preprint} arXiv:1805.00918, 2018.
\bibitem{Brenner17-2}
S.C. Brenner, L.-Y. Sung. Virtual element methods on meshes with small edges or faces, {\em Mathematical Models and Methods in Applied Sciences}, \textbf{28}, (2018): 1291--1336.
\bibitem{V11}
V. Maz'ya. Sobolev Spaces with Applications to Elliptic Partial Differential Equations. {\em Springer}, Heidelberg, 2011.
\bibitem{J87}
J. Wloka. Partial Differential Equations. {\em Cambridge University Press}, Cambridge, 1987.
\end{thebibliography}
\end{document}